\documentclass[12pt]{amsart}
\usepackage{amsmath,amssymb,amsthm}
\usepackage{mathtools}
\usepackage{xcolor}
\usepackage{graphicx} 
\usepackage{pgf,tikz}
\usetikzlibrary{arrows}
\usepackage{lipsum}
\usepackage{caption}  
\usepackage{subcaption}
\usepackage{pgfplots}
\usepackage{scalerel}
\usepackage{amsmath}
\usepackage{amssymb}
\usepackage{amsthm}
\usepackage{enumerate}
\usepackage{hyperref}
\hypersetup{colorlinks=true,citecolor=blue,linkcolor=red}
\usepackage{cite}
\usepackage{color}
\usepackage{bbm}
\usepackage[usestackEOL]{stackengine}

\usepackage[a4paper,width=16.5cm,top=3cm,bottom=2.8cm]{geometry}

\newtheorem{theorem}{Theorem}[section]
\newtheorem{proposition}[theorem]{Proposition}
\newtheorem{corollary}[theorem]{Corollary}
\newtheorem{lemma}[theorem]{Lemma}
\newtheorem{definition}[theorem]{Definition}
\newtheorem{remark}[theorem]{Remark}

\theoremstyle{plain}
\newtheorem*{theorem*}{Theorem}

\def\dashint{\,\ThisStyle{\ensurestackMath{%
			\stackinset{c}{.2\LMpt}{c}{.5\LMpt}{\SavedStyle-}{\SavedStyle\phantom{\int}}}%
		\setbox0=\hbox{$\SavedStyle\int\,$}\kern-\wd0}\int}
\newcommand{\rn}{\mathbb{R}^n}

\newcommand{\bmorn}{{\text{BMO}}(\rn)}

\newcommand{\lonern}{{{L}}^1(\mathbb{R}^n)}

\newcommand{\lonernn}[1]{\|#1\|_{\lonern}}

\newcommand{\lprn}{{{L}}^p(\mathbb{R}^n)}

\newcommand{\linftyrn}{{{L}}^{\infty}(\mathbb{R}^n)}

\newcommand{\linftyrnn}[1]{\|#1\|_{\linftyrn}}

\newcommand{\lonernloc}{{{L}}_{{\text{loc}}}^1(\mathbb{R}^n)}

\newcommand{\ave}[2]{\dashint_{#2}{#1}}

\newcommand{\osc}[2]{\text{osc}(#1,#2)}

\newcommand{\jnp}[1]{\|#1\|_{{\text{JN}}_p(\rn)}}
\newcommand{\jnprn}{{\text{JN}}_p(\rn)}
\newcommand{\upperhalf}{\mathbb{R}_+^{n+1}}

\begin{document}
	
	\title{A view from above on $\jnprn$}
	\author{Shahaboddin Shaabani}
	\address{Department of Mathematics and Statistics\\Concordia University}
	\email{shahaboddin.shaabani@concordia.ca}
	
\begin{abstract}
For a symmetric convex body $K\subset\rn$ and $1\le p<\infty$, we define the space $S^p(K)$ to be the tent generalization of $\jnprn$, i.e., the space of all continuous functions $f$ on the upper-half space $\upperhalf$ such that
\[
\|f\|_{S^p(K)} := \big( \sup_{\mathcal{C}} \sum_{B \in \mathcal{C}} |f_B|^p \big)^{\frac{1}{p}} < \infty,
\]
where, in the above, the supremum is taken over all finite disjoint collections of homothetic copies of $K$. It is then shown that the dual of $S^1_0(K)$, the closure of the space of continuous functions with compact support in $S^1(K)$, consists of all Radon measures on $\upperhalf$ with uniformly bounded total variation on cones with base $K$ and vertex in $\rn$. In addition, a similar scale of spaces is defined in the dyadic setting, and for $1\le p<\infty$, a complete characterization of their duals is given. We apply our results to study dyadic $\text{JN}_p$ spaces.

\end{abstract}	
\keywords{$\text{JN}_p$ space, Tent spaces, Intersection graphs }
\subjclass{42B35 }
\date{}
\maketitle
\section{Introduction}	
	
For $1<p<\infty$, the space $\jnprn$ consists of all functions $f\in \lonernloc$ such that  
\begin{equation}\label{jnpdef}		
	\jnp{f}:=\big(\sup_{\mathcal{C}}\sum_{Q\in\mathcal{C}}\osc{f}{Q}^p|Q|\big)^{\frac{1}{p}}<\infty,
\end{equation}
where, in the above, $Q$ denotes a closed cube with sides parallel to the axes, $\mathcal{C}$ is a finite collection of disjoint closed cubes, and $\osc{f}{Q}$ is the mean oscillation of $f$ over $Q$, i.e.,  
\[
\osc{f}{Q}:=\ave{|f-f_Q|}{Q}, \quad f_Q:=\ave{f}{Q}.
\]
Throughout the paper, we use $|E|$ for the Lebesgue measure of a measurable set $E$, and $\ave{f}{E}:=\frac{1}{|E|}\int_{E}f$ for the average of $f$ over $E$.\\

Alongside the well-known $\text{BMO}$-norm, the above cumbersome norm was introduced by F. John and L. Nirenberg in \cite{MR0131498}. It is not hard to see that $\lprn$ is embedded (continuously) in $\jnprn$. Furthermore, in \cite{MR0131498}, after the authors proved their famous lemma, it was shown that $\jnprn$ is embedded in ${L}^{p,\infty}(\rn)$, the weak-$\lprn$. In addition, the above norm tends to the $\text{BMO}$-norm as $p$ goes to infinity. Therefore, it is tempting to think that $\jnprn$ could play the role of $\bmorn$, $\lprn$ the role of $\linftyrn$, and the embedding $\jnprn\subset {L}^{p,\infty}(\rn)$ could be a variant of the John-Nirenberg inequality. However, this analogy has not gone any further than this since, despite the crucial role of $\bmorn$ in harmonic analysis, up until now, only a few applications of $\jnprn$ have been known.\\

Nevertheless, in recent years, there has been an interest in understanding various aspects of this space, and many papers have been written in this direction \cite{MR3806819,MR4686687,MR4546791,MR4323004,MR4480231,MR2794938}. However, it is not easy at all to work with this norm, and constructing functions in $\jnprn\backslash\lprn$ has proved to be difficult \cite{MR3806819}. One source of this difficulty lies in the appearance of sums over disjoint collections, and the other stems from the mean oscillation. In fact, changing the latter factor can drastically change the meaning of the norm. For instance, replacing the mean oscillation with the average over cubes results in the $\lprn$ norm \cite{MR3806819}. Indeed, there are various norms appearing in approximation theory that have definitions similar to \eqref{jnpdef} but use quantities different from the mean oscillation \cite{MR4078209,MR4040115}. In addition, there are norms using other convex bodies, such as balls, which are defined in a similar way \cite{MR2794938,MR4376642}.\\

The purpose of the present paper is to study those properties of such norms that are shared for every choice of these quantities and depend only on the structure of these norms. To this aim, we appeal to the idea of tent spaces, which were introduced by R. Coifman, Y. Meyer, and E. Stein. In \cite{MR791851}, the authors showed that some phenomena, such as ``${H}^1-\text{BMO}$'' duality or atomic decomposition, hold in a very general setting and, to a large extent, are consequences of the geometric structure of the norms. In \cite{MR791851}, tent spaces and many of their properties, including their duals, were studied, but a scale of such spaces, denoted there by $T_p^{\infty}(\rn)$, was considered to be ill-defined and therefore was almost excluded from their study (see the remark in Section 3 in \cite{MR791851}). The reason for this lies in the authors' interest in the structure of the $\text{BMO}$-norm, but $T_p^{\infty}(\rn)$ is a generalization of the space of functions with a bounded square function and not $\bmorn$. Therefore, in \cite{MR791851}, their definition was modified, and the scale of $T_p^{\infty}(\rn)$ spaces was abandoned (see remark (a) of Section 6 in \cite{MR791851}). Here, we will show that there is a close connection between these spaces and $\jnprn$.\\

Here, we fix a symmetric convex body $K\subset\rn$ and, for $1\le p<\infty$, define a scale of spaces denoted by $S^p(K)$, which are ``tent generalizations'' of $\text{JN}_p$-type spaces. Then, in Theorem \ref{duality1}, we show that the predual of the space $T_1^\infty(\rn)$, which we modify and denote by $T^1(K)$, is $S_0^1(K)$, the closure of the space of continuous functions with compact support in $S^1(K)$. At the heart of this duality lies the problem of partitioning a family of homothetic copies of $K$ with some finite total overlap into disjoint subfamilies whose number is bounded by the total overlap of the family. Fortunately, such problems have been studied in combinatorial convex geometry, and here we apply a simple result in this direction \cite{MR125491,MR2065259,MR2097318,MR2807079,MR2971748}. The only place in harmonic analysis where we have seen such a problem is in the proof of the covering theorem of Besicovitch \cite{MR1249875,MR3409135}, and Theorem \ref{duality1} shows that such problems are related to the study of $\text{JN}_p$-type norms. It is unfortunate that for $1<p<\infty$, we could not successfully characterize the dual of the space $S_0^p(K)$, and in Theorem \ref{preduals}, we show that the predual of $T^p(K)$ is totally different from $S_0^p(K)$. However, in the dyadic setting, we give a complete characterization of the dual of $S_0^p(K)$-type spaces and show that even in this simple setting, the dual norms are complicated (see Theorem \ref{xdual}).

\section{The Tent Generalization of $\jnprn$}
To begin with, let us fix some notation. Throughout the paper, we fix a symmetric convex body $K\subset \rn$, which is a compact, convex set with nonempty interior and is symmetric about the origin, i.e., when $x\in K$, we have $-x\in K$. For such a set, we denote the associated norm by $\|.\|_{K}$ or, when it is clear from the context, by $\|.\|$, and the closed ball (with respect to this norm) centered at $x\in\rn$ with radius $t$ by $B_K(x,t)$ or $B(x,t)$. These balls will be referred to as $K$-balls, and we use $B$ to denote such a ball when its center and radius are not of interest to us. We also use $\upperhalf$ for the upper half-space of $\rn$, i.e., $\rn\times \mathbb{R}_+$, and for a symmetric convex body $K$ and $x\in\rn$, $\Gamma_x$ is used for the cone with $K$ as its base and $x$ as its vertex, that is,
\[
\Gamma_x:=\left\{(y,t)\in\upperhalf:\|x-y\|_K\le t \right\}.
\]
In addition, ${C}_b(\upperhalf)$ and ${C}_c(\upperhalf)$ are used for the spaces of bounded continuous and compactly supported continuous functions on the upper half-space, respectively. Also, since we regard a point $(x,t)\in\upperhalf$ as the $K$-ball, $B=B(x,t)$, for a function $f\in C_b(\upperhalf)$, we sometimes use $f_B$ instead of $f(x,t)$. Finally, we use the standard notation $\lesssim$, $\gtrsim$, and $\simeq$ to suppress all dimensional constants that might appear in inequalities. Having fixed some notation, we give the tent space generalization of the norm defined in \eqref{jnpdef}.
\begin{definition}
	For a symmetric convex body $K$ and $1\le p<\infty$, the space ${S}^p(K)$ consists of all continuous functions on $\upperhalf$ such that
	\begin{equation}
		\|f\|_{S^p(K)} := \big( \sup_{\mathcal{C}} \sum_{B \in \mathcal{C}} |f_B|^p \big)^{\frac{1}{p}} < \infty,
	\end{equation}
	where the supremum is taken over all finite collections of pairwise disjoint $K$-balls.
\end{definition}
The above space is embedded in $C_b(\upperhalf)$ and is complete. Indeed, given a Cauchy sequence $\left\{f_m\right\}$ in $S^p(K)$, we have that it is uniformly convergent to a function $f$, and thus for a given collection of pairwise disjoint $K$-balls, $\mathcal{C}$, we have
\[
\sum_{B \in \mathcal{C}} |f_m - f|_B^p = \lim_{l \to \infty} \sum_{B \in \mathcal{C}} |f_m - f_l|_B^p \le \liminf_{l \to \infty} \|f_m - f_l\|_{S^p(K)}^p,
\]
which implies that $\lim_{m \to \infty} \|f_m - f\|_{S^p(K)} = 0$, which is the desired result.\\
The space $S^p(K)$ contains $C_c(\upperhalf)$ and functions of the form
\[
f(x,t) = |B(x,t)|^{\frac{1}{p}} \ave{g}{B(x,t)}, \quad g \in \lprn,
\]
which can be verified by Hölder's inequality. In fact, the above formula gives an isometric embedding of $\lprn$ into $S^p(K)$. A more interesting example comes from functions in $\jnprn$. That is, for $1 < p < \infty$, and a function $g \in \jnprn$, by choosing $Q_0 = [-1,1]^n$ and
\[
f(x,t) = \osc{g}{B_{Q_0}(x,t)} |B_{Q_0}(x,t)|^{\frac{1}{p}},
\]
we have that $\jnp{g} = \|f\|_{S^p(Q_0)}$.\\

Later, we deal with the space $S^1_0(K)$, the closure of $C_c(\upperhalf)$ in $S^1(K)$. Unfortunately, we could not find a simple characterization of this space. However, the next proposition provides some non-trivial examples of functions belonging to $S^1_0(K)$.

\begin{proposition}\label{stripproposition}
	Every function $f \in S^1(K)$ with support contained in a strip of the form $\rn \times [a,b] \subset \upperhalf$ belongs to $S^1_0(K)$.
\end{proposition}

\begin{proof}
Here, the crucial observation is that for every $\varepsilon > 0$, if $M$ is chosen large enough, we must have
\begin{equation}\label{smalout}
	\sup_{\cup \mathcal{C} \subset B(0,M)^c} \sum_{B \in \mathcal{C}} |f_B|^p \le \varepsilon,
\end{equation}
where the supremum is taken over all disjoint families of $K$-balls lying outside of $B(0,M)$. Indeed, if this were not the case, we could choose a sequence of disjoint collections $\{\mathcal{C}_i\}_{i \ge 1}$ of balls, far from each other, such that
\[
\sum_{B \in \mathcal{C}_i} |f_B|^p > \varepsilon, \quad i \ge 1,
\]
and their union $\cup \mathcal{C}_i$ would also be disjoint — a contradiction to the fact that $f \in S^p(K)$. 

Now, choose $M > b$ sufficiently large so that \eqref{smalout} holds. Then, take a continuous function $\varphi \in C_c(\upperhalf)$ with $0 \le \varphi \le 1$ such that
\[
\varphi(x,t) = 1, \quad (x,t) \in B(0,2M) \times [a,b].
\]
Note that the function $(1-\varphi)f$ is supported in $B(0,2M)^c \times [a,b]$. Therefore, if $(1-\varphi)f(x,t) \neq 0$, we must have $B(x,t) \subset B(0,M)^c$, which, together with \eqref{smalout}, implies
\[
\|(1-\varphi)f\|_{S^p(K)} \le \varepsilon^{\frac{1}{p}}.
\]
This proves that $f$ belongs to $S^1_0(K)$.
\end{proof}
Another observation worth mentioning is that for $1 \le p < \infty$, the non-tangential maximal operator
\[
M_p(f)(x) := \sup_{(y,s) \in \Gamma_x} \frac{|f(y,s)|}{|B(y,s)|^{\frac{1}{p}}},
\]
maps $S^p(K)$ into $L^{p,\infty}(\mathbb{R}^n)$; that is,
\[
\|M_p(f)\|_{L^{p,\infty}(\mathbb{R}^n)} \lesssim \|f\|_{S^p(K)}.
\]
To see this, let $\lambda > 0$ and let $F$ be a compact subset of $\{M_p(f) > \lambda\}$. For each $x \in F$, choose a $K$-ball $B(y,s)$ containing $x$ such that
\[
|B(y,s)| \le \frac{1}{\lambda^p} |f(y,s)|^p.
\]
By Vitali's covering lemma, we obtain a disjoint collection of $K$-balls $\mathcal{C}$ such that
\[
|F| \le \frac{3^n}{\lambda^p} \sum_{B(y,s) \in \mathcal{C}} |f(y,s)|^p \le \frac{3^n}{\lambda^p} \|f\|_{S^p(K)}^p,
\]
which proves the claim.\\

We proceed by recalling the definition of tent spaces, denoted by $T_p^{\infty}(\rn)$ in \cite{MR791851}, which we simplify here and denote by $T^p(K)$.

\begin{definition}
	For a symmetric convex body $K$, the space $T^1(K)$ consists of all Radon measures $\mu$ on $\upperhalf$ such that 
	\begin{equation}
		\|\mu\|_{T^1(K)} := \linftyrnn{A_1(\mu)} < \infty, \quad A_1(\mu)(x) := |\mu|({\Gamma_x}).
	\end{equation}
	In a similar way, for $1 < p < \infty$, the space $T^p(K)$ is defined as the space of all locally integrable functions $f$ on $\upperhalf$ such that
	\begin{equation}
		\|f\|_{T^p(K)} := \linftyrnn{A_p(f)} < \infty, \quad A_p(f)(x) := \|f\|_{L^p(\Gamma_x)}.
	\end{equation}
\end{definition}
As promised earlier, in the next theorem, we connect the study of $\text{JN}_p$-type norms to the above tent spaces.

\begin{theorem}\label{duality1}
	For every symmetric convex body $K \subset \rn$, the space $T^1(K)$ is the dual of $S_0^1(K)$, the closure of $C_c(\upperhalf)$ in $S^1(K)$. More precisely, any $\mu \in T^1(K)$ defines a bounded linear functional on $S_0^1(K)$, and its action is given by integration with respect to $\mu$. Moreover, the norm of the functional is comparable to $\|\mu\|_{T^1(K)}$. Conversely, for any continuous linear functional $l$ on $S_0^1(K)$, there exists a unique measure $\mu \in T^1(K)$ such that
	\[
	l(f) = \int_{\upperhalf} f \, d\mu, \quad f \in S_0^1(K), \quad  \|l\|_{S_0^1(K)^*} \simeq \|\mu\|_{T^1(K)}.
	\]
	The above integral is absolutely convergent, and all implied constants depend only on the dimension.
\end{theorem}
Unfortunately, the above duality does not extend to other values of $p$, since for any point $(x,t) \in \upperhalf$, the point mass measure $\delta_{(x,t)}$ defines a bounded linear functional on $S_0^p(K)$. However, as the next theorem shows, the predual of $T^p(K)$ lies in $L_{\text{loc}}^{p'}(\upperhalf)$. Moreover, as Theorem \ref{xdual} demonstrates, for $1 < p < \infty$, even in the dyadic setting, the dual of $S_0^p(K)$ is much more complicated than that of $S_0^1(K)$.

\begin{theorem}\label{preduals}
	For a given symmetric convex body $K \subset \rn$ and $1 < p < \infty$, the space $T^p(K)$ is a dual space, and its predual $U^{p'}(K)$ consists of functions $f \in L_{\text{loc}}^{p'}(\upperhalf)$ for which the following atomic decomposition holds:
	\begin{align}
		&f = \sum_{i=1}^{\infty} \lambda_i f_i, \quad \|f\|_{U^{p'}(K)} \simeq \sum_{i=1}^{\infty} |\lambda_i|, \\
		&f_i(x,t) = a_i(x,t) \int_{B(x,t)} g_i(y) \, dy, \quad \|g_i\|_{L^1(\rn)} = 1, \quad \|a_i\|_{L^{p'}(\upperhalf)} \le 1, \quad i \ge 1.
	\end{align}
In the above, the convergence is understood in the norm topology of $U^{p'}(K)$, and therefore in $L_{\text{loc}}^{p'}(\upperhalf)$. \\

Furthermore, the above decomposition holds for functions in $S_0^1(K)$, the predual of $T^1(K)$, with the modification that in this case, $a_i \in C_c(\upperhalf)$ with $|a_i(x,t)| \le 1$ for all $(x,t) \in \upperhalf$.
\end{theorem}

As a corollary of Theorem \ref{preduals} we have:	
\begin{corollary}
	For $1<p<\infty$, any function $f \in \jnprn$, and $0 < \delta < 1$, there exist three sequences of positive numbers, $\left\{\lambda_i\right\}_{i\ge1}$, integrable functions $\left\{g_i\right\}_{i\ge1}$, and compactly supported continuous functions $\left\{a_i\right\}_{i\ge1}$ such that for every cube $Q$ with side length $\delta < l(Q) < \delta^{-1}$, we have
	\begin{align}
		& \osc{f}{Q}^p = \sum_{i=1}^{\infty} \lambda_i a_i(Q) \ave{g_i}{Q}, \quad \sum_{i=1}^{\infty} \lambda_i \lesssim \|f\|_{\jnprn}^p, \\
		& \|g_i\|_{L^1(\rn)} = 1, \quad |a_i| \le 1, \quad a_i(Q) := a_i(x,t), \quad Q = x + tQ_0, \quad i \ge 1.
	\end{align}
	In the above, the convergence is understood in $S^1(K)$ and thus uniform on the entire upper-half space.
\end{corollary}

\begin{proof}
	Take a continuous function $\varphi$ on $\mathbb{R}_+$ such that $\varphi = 1$ on $[\delta, \delta^{-1}]$, and supported on an interval $[a,b]\subset\mathbb{R}_+$. Then, note that the function
	\[
	g(x,t) = \varphi(t) \osc{f}{B_{Q_0}(x,t)}^p |B_{Q_0}(x,t)|,
	\]
	belongs to $S^1(Q_0)$ and vanishing outside of the strip $\rn\times\left[a,b\right]$. Therefore, the claim follows from Theorem \ref{preduals} and Proposition \ref{stripproposition}.
\end{proof}

Now, we proceed to the proof of Theorem \ref{duality1}, and to this end, we recall some notions of graph theory. A graph $G=(V,E)$ is a finite simple graph, i.e., a finite set of vertices $V$ and a set of edges $E$ joining pairs of elements in $V$ such that any two vertices are connected by at most one edge. A graph $H=(V',E')$ is called an induced subgraph of $G$ if $V' \subset V$, $E' \subset E$, and $E'$ contains all the edges between elements of $V'$ in $G$. In a graph $G=(V,E)$, the degree of each vertex is the number of vertices adjacent to it. A graph is called $k$-degenerate if every induced subgraph of $G$ has a vertex with degree at most $k$. A proper $k$-coloring of a graph $G$ is a way of coloring each vertex of $G$ with $k$ colors, such that no two adjacent vertices share the same color. A graph is called $k$-colorable if such a coloring exists, and the minimum number of colors required for a proper coloring of $G$ is referred to as the chromatic number of $G$ and is denoted by $\chi(G)$. Here, we will use the simple fact that every $k$-degenerate graph is $(k+1)$-colorable, that is, $\chi(G) \le k+1$ \cite{MR1633290}.\\

Given a sequence of sets $S=\{S_i\}_{i\ge1}$, there is a natural graph, called the intersection graph of $S$, which captures the intersection structure of $S$ (note that by a sequence we mean that it is allowed to have $S_i=S_j$ with $i\neq j$). For any set in the sequence, consider a vertex and join two vertices by an edge whenever their intersection is non-empty. Here, we will be dealing with the intersection graphs of finite sequences of $K$-balls. The problem that we will face is to bound the chromatic number of the intersection graph of a sequence of $K$-balls, $S=\{B_i\}_{i\ge1}$, with its total overlap defined as
\[
\textbf{O}(S):=\max_{x\in \rn}\sum_{i\ge1}1_{B_i}(x).
\]
As mentioned before, these kinds of problems have been studied in combinatorial convex geometry, and a major tool here is the notion of a transversal set or a piercing set. Given a collection of distinct objects, $\mathcal{C}$, a set is called transversal to $\mathcal{C}$ if every set in $\mathcal{C}$ contains a point from that set, and the cardinality of a minimal such set is called the transversal or piercing number of $\mathcal{C}$ and is denoted by $\tau(\mathcal{C})$. Below, we present an interesting simple result in this direction, which we use shortly.
\begin{theorem}\label{transversalnumber}
	Let $\mathcal{C}$ be a finite collection of distinct $K$-balls, $B_1$ be one with the smallest volume, and $\mathcal{C}_1$ be the sub-collection of all $K$-balls in $\mathcal{C}$ that intersect $B_1$. Then, there exists a dimensional constant $c(n)$ such that $\tau(\mathcal{C}_1)\le c(n)$. 
\end{theorem}

See Theorem 1 in \cite{MR125491}, or Lemma 5 in \cite{MR2971748} for a proof. One particular application of the above theorem is that a finite collection of $K$-balls, $\mathcal{C}$, can be decomposed into $c(n) \textbf{O}(\mathcal{C}) + 1$ disjoint sub-collections. In the next lemma, we prove a slight generalization of this result for a finite sequence of $K$-balls.

\begin{lemma}\label{sequenceoverlap}
	Let $\mathcal{C}$ be a finite collection of distinct $K$-balls and $S=\{B_i\}_{i\ge1}$ a finite sequence of elements in $\mathcal{C}$. More precisely, for each $i\ge1$, $B_{i}\in \mathcal{C}$, and it might be the case that $B_{i}=B_{j}$ with $i\neq j$. Then, for $G$ the intersection graph of $S$, we have
	\[
	\chi(G)\le c(n) \textbf{O}(S)+1.
	\]
\end{lemma}    

\begin{proof}
	Without loss of generality, assume $B_1$ has the smallest volume. Then, from Theorem \ref{transversalnumber}, there exists a set of points $F$ with cardinality at most $c(n)$ such that every $K$-ball in $S$ intersecting $B_1$ contains a point of $F$. Now, for $x\in F$, the number of elements in $S$ containing $x$ is at most $\textbf{O}(S)$. Therefore, the number of $K$-balls in the sequence intersecting $B_1$ is at most $c(n)\textbf{O}(S)$, which implies that the intersection graph of $S$ is $c(n)\textbf{O}(S)$-degenerate and thus it must be $c(n)\textbf{O}(S)+1$-colorable. The proof is complete.
\end{proof}

\begin{remark}
In \cite{MR2065259,MR2097318,MR2807079}, the upper bounds for the chromatic number are given in terms of the clique number of the graph. Consequently, these bounds, together with the above lemma, imply that for the intersection graph of a finite collection of \( K \)-balls, the chromatic number, clique number, and total overlap are all comparable.

\end{remark} 

Now, with Lemma \ref{sequenceoverlap} in hand, we prove a discrete version of Theorem \ref{duality1}.

\begin{lemma}\label{descreetduality}
	Let $\mathcal{C}$ be a finite collection of $K$-balls, $f=\{f_B\}_{B \in \mathcal{C}}$, and $g=\{g_B\}_{B \in \mathcal{C}}$ two sequences of real numbers indexed by elements in $\mathcal{C}$. Also, let
	\[
	\|f\|_{S^1(\mathcal{C})}:=\max_{\mathcal{C}'}\sum_{B\in\mathcal{C}'}|f_{B}|, \quad \|g\|_{T^1(\mathcal{C})}:=\max_{x\in \rn}\sum_{B\in\mathcal{C}}|g_{B}|1_{B}(x),
	\]
	where the maximum on the left is taken over disjoint sub-collections of $\mathcal{C}$. Then, we have
	\begin{align}\label{descreetduality1}
		&|\sum_{B\in\mathcal{C}} f_B g_B|\lesssim \|f\|_{S^1(\mathcal{C})}\|g\|_{T^1(\mathcal{C})}.
	\end{align}
\end{lemma}

\begin{proof}
	After using the triangle inequality, we may assume that $f_B$ and $g_B$ are non-negative numbers for $B \in \mathcal{C}$. Then, assume that the coefficients $g_B$ are all non-negative integers and $\|g\|_{T^1(\mathcal{C})} = N$. Now, form a finite sequence $\{B_i\}_{i \ge 1}$ consisting of members of $\mathcal{C}$ such that every $B \in \mathcal{C}$ appears exactly $g_B$ times in $\{B_i\}_{i \ge 1}$, and whenever $B_i = B$, set $f_{B_i} = f_B$. Now, it follows from Lemma \ref{sequenceoverlap} that we may decompose this sequence into $2c(n)N$ sub-collections $\{\mathcal{C}_j': 1 \le j \le 2c(n) N\}$, such that each consists only of disjoint $K$-balls in $\mathcal{C}$, and thus we have
	\[
	\sum_{B \in \mathcal{C}} f_B g_B = \sum_{i} f_{B_i} = \sum_{1 \le j \le 2c(n) N} \sum_{B \in \mathcal{C}_j'} |f_B| \lesssim N \|f\|_{\mathcal{C}} = \|f\|_{S^{1}(\mathcal{C})} \|g\|_{T^1(\mathcal{C})}.
	\]
	Now, if the coefficients $g_B$ are all rational numbers, we may find a common denominator $M$, multiply the coefficients of $g$ by $M$, use the above inequality, and then divide both sides by $M$ to get the desired inequality. Finally, when the coefficients of $g$ are real numbers, we may approximate them from below by rationals and then take the limit, and this completes the proof.
\end{proof}

To see why the above lemma is the discrete version of Theorem \ref{duality1}, for a finite collection of $K$-balls $\mathcal{C}$ and $g \in T^1(\mathcal{C})$, set
\[
\mu = \sum_{B(y,t) \in \mathcal{C}} g_{B(y,t)} \delta_{(y,t)}.
\]
Then, note that
\[
\int_{\upperhalf} f \, d\mu = \sum_{B \in \mathcal{C}} f_B g_B,
\]
and
\[
A_1(\mu)(x) = |\mu|(\Gamma_x) = \sum_{B \in \mathcal{C}} |g_B| 1_{B}(x), \quad \|g\|_{T^1(\mathcal{C})} = \max_{x \in \rn} A_1(\mu)(x).
\]
Therefore, the next lemma completes the claim.

\begin{lemma}\label{supmax}
	For $\mu$ a non-negative Radon measure on $\upperhalf$, we have
	\[
	\sup_{x \in \rn} A_1(\mu)(x) \simeq \|\mu\|_{T^1(K)}.
	\]
\end{lemma}

\begin{proof}
Fix $x \in \rn$ and $t > 0$, and note that Fubini's theorem implies
\begin{align}
	&\ave{A_1(\mu)(y)}{B(x,t)} dy = \ave{\int_{\Gamma_y} d\mu(z,r)}{B(x,t)} dy = \int_{\upperhalf} \frac{|B(z,r) \cap B(x,t)|}{|B(x,t)|} d\mu(z,r).
\end{align}
Therefore, applying Fatou's lemma, we get
\[
\int_{\Gamma_x} \liminf_{t \to 0^+} \frac{|B(z,r) \cap B(x,t)|}{|B(x,t)|} d\mu(z,r) \le \|\mu\|_{T^1(K)}.
\]
Now, for a ball $B(z,r)$ containing $x$ as an interior point, the limit inside the integral is $1$. It remains to show that for $x$ on the boundary of $B(z,r)$, we have
\[
\liminf_{t \to 0^+} \frac{|B(z,r) \cap B(x,t)|}{|B(x,t)|} \gtrsim 1, \quad x \in \partial B(z,r), \quad z \in \rn, \quad r > 0.
\]

To this end, we appeal to John's ellipsoid of $K$, i.e., the ellipsoid $E$ with the property that $E \subset K \subset \sqrt{n}E$. Now, since the claim remains unchanged by an invertible linear transformation of $\rn$, we may assume that $E$ is the Euclidean unit ball of $\rn$. Furthermore, we may assume that $x = 0$, $r = 1$, and $z = (0, 0, 0, \dots, d)$. Now, since the closed convex hull of $z + E \cup \{x\}$ lies in $B(z,1)$, we conclude that the Euclidean truncated cone
$$C_d = \{(y', y_n) \in \rn : \|y'\|_E \le \frac{1}{d}y_n, \quad y_n \le d\},$$
lies in $B(z, 1)$. Therefore, since $B(z, 1) \subset z + \sqrt{n}E$, we get $d \le \sqrt{n}$ and thus 
$$C_n = \{(y', y_n) \in \rn : \|y'\|_E \le \frac{1}{\sqrt{n}}y_n, \quad y_n \le d\} \subset B(z, 1),$$
from which we conclude
\[
\frac{|B(z,1) \cap B(0,t)|}{|B(0,t)|} \ge \frac{|C_n \cap tE|}{|\sqrt{n}tE|} = \frac{1}{\sqrt{n}^n} \frac{|C_n \cap tE|}{|tE|}.
\]
Now, dilation implies that for $t < d$, the expression on the right-hand side is a constant that depends only on the dimension, and this completes the proof.
\end{proof}

As a simple corollary of the above lemma, we have
\begin{corollary}\label{almosteveryoverlap}
	For a finite collection of $K$-balls, $\mathcal{C}$, and any non-negative sequence $\{f_B\}_{B \in \mathcal{C}}$, we have
	\[
	\linftyrnn{\sum_{B \in \mathcal{C}} f_B 1_B} \simeq \max_{x \in \rn} \sum_{B \in \mathcal{C}} f_B 1_B(x).
	\]
	In particular,
	\[
	\textbf{O}(\mathcal{C}) \simeq \linftyrnn{\sum_{B \in \mathcal{C}}1_B}.
	\]
\end{corollary}
\begin{remark}
The above corollary implies that the total overlap of a finite collection of closed $K$-balls is comparable to the total overlap of their interiors. Therefore, by Lemma \ref{sequenceoverlap}, there exists a dimensional constant \( c(n) \) such that any finite collection of $K$-balls with disjoint interiors can be partitioned into at most \( c(n) \) sub-collections whose closures are disjoint. Consequently, replacing closed $K$-balls with open ones in the definition of \( S^p(K) \) results in a comparable norm.
\end{remark}

Now, as the final ingredient, we bring the following lemma.

\begin{lemma}\label{gammadelta}
Let $0<\delta<\frac{a}{2}$, $x\in\rn$ and 
\begin{equation}\label{gammadeltadef}
	\Gamma_{x,a,\delta}=\left\{(y,s)\in\upperhalf: \|y-x\|_K\le\delta+ s,\quad a\le s \right\}.
\end{equation}
Then for any non-negative measure $\mu \in T^1(K)$ we have
\[
\mu(\Gamma_{x,a,\delta})\lesssim \|\mu\|_{T^1(K)}.
\]	
\end{lemma}
\begin{proof}
It is enough to show that for at most $3^n$ points $\{x_i: 1\le i\le 3^n\}$, we have
\begin{equation}\label{conecover}
	\Gamma_{x,a,\delta}\subset \bigcup_{1\le i\le 3^n} \Gamma_{x_i},
\end{equation}
which, together with Lemma \ref{supmax}, proves the claim. Now, to see why \eqref{conecover} holds, without loss of generality, assume $x=0$, and let $\{x_i\}$ be a maximal set of $\frac{a}{2}$-separated points in $B(0,a)$, i.e., a maximal set of points with $\|x_i-x_j\|_K\ge \frac{a}{2}$. Note that the cardinality of such a set does not exceed $3^n$, and for $(y,s)\in \Gamma_{0,a,\delta}$, we have $y'=a(s+\delta)^{-1}y\in B(0,a)$. So, for at least one $x_i$, we must have $\|y'-x_i\|_K<\frac{a}{2}$, which implies that
\[
\|y-x_i\|_K\le \|y'-x_i\|_K + \|y'-y\|_K< \frac{a}{2}+(1-\frac{a}{s+\delta})(s+\delta)< s,
\]
which means that $(y,s)\in \Gamma_{x_i}$, and this proves \eqref{conecover}.

\end{proof}
\begin{remark}
It is worth noting that in the above proof, we actually covered \(\Gamma_{x,a,\delta}\) with the interior of the cones \(\Gamma_{x_i}\). This implies that if, in the definition of the \(T^1(K)\) norm, we replace the closed cones with their interiors, we obtain an equivalent norm. The reason is that by keeping \(a > 0\) fixed and letting \(\delta\) tend to \(0\), we obtain a bound for  
\[
\mu\left\{(y,s)\in\upperhalf: \|y-x\|\le s,\quad 0<a\le s \right\}
\]
using this seemingly weaker norm. Then, letting \(a\) tend to \(0\) establishes the equivalence.
	
\end{remark}
Now, we prove Theorem \ref{duality1}.
\begin{proof}[Proof of Theorem \ref{duality1}] 
First, we show that every $\mu\in T^1(K)$ defines a bounded linear functional on $S^1(K)$. So, let $\mu$ be a nonnegative measure in $T^1(K)$ and $f$ a non-negative function in $S^1(K)$. Then, we have to show that  
\begin{equation}\label{thefirstdirection}
	\int_{\upperhalf}fd\mu\lesssim \|f\|_{S^1(K)}\|\mu\|_{T^1(K)}.
\end{equation}

To do this, let $a>0$ and consider the region $F=[-a^{-1},a^{-1}]^n\times [a,a^{-1}]$. Then, let $\varepsilon>0$, and since $f$ is continuous, choose $0<\delta<4^{-1}a$ such that  
\[
|f(x,t)-f(y,s)| \le\varepsilon, \quad \text{for} \quad \|x-y\|+|t-s|\le\delta, \quad \text{with} \quad (x,t),(y,s)\in F.
\]
Next, partition $F$ into disjoint rectangular pieces $R$ for which $$\|x-y\|+|t-s|\le\delta \quad \text{when} \quad (x,t),(y,s)\in R,$$  
and from each piece $R$, choose an arbitrary point $(x_R,t_R)$. Now, we have  
\begin{align}
	\int_{F}fd\mu=\sum_{R} \int_{R}fd\mu&\le\sum_{R} f(x_R,t_R)\mu(R)+\sum_{R}\int_{R}|f-f(x_R,t_R)|d\mu\\
	&\le\sum_{R} f(x_R,t_R)\mu(R)+\varepsilon \mu(F).\label{firstestimate}
\end{align}
Next, let $\mathcal{C}=\{B(x_R,t_R)\}$, and note that Lemma \ref{descreetduality} implies that  
\begin{equation}\label{the second}
	\sum_{R} f(x_R,t_R)\mu(R)\lesssim \|f\|_{S^1(K)}
	\max_{x\in \rn}\sum_{R}\mu(R)1_{B(x_R,t_R)}(x).
\end{equation}

Also, for $x\in\rn$ we have  
\[
\sum_{R}\mu(R)1_{B(x_R,t_R)}(x)=\sum_{x\in B(x_R,t_R)}\mu(R),
\]
which, after recalling that the sets $R$ are disjoint and  
\[
\|y-x\|\le \|y-x_R\|+\|x_R-x\|\le\delta+t_R\le 2\delta+ s, \quad (y,s)\in R,
\]
gives us  
\begin{equation}\label{thethird}
	\sum_{R}\mu(R)1_{B(x_R,t_R)}(x)=\sum_{x\in B(x_R,t_R)}\mu(R)\le \mu (\Gamma_{x,a,2\delta}),
\end{equation}
where $\Gamma_{x,a,2\delta}$ is as in \eqref{gammadeltadef}. So, from \eqref{firstestimate}, \eqref{the second}, \eqref{thethird}, and Lemma \ref{gammadelta}, we get  
\begin{equation*}
	\int_{F}fd\mu\lesssim \|f\|_{S^1(K)}\|\mu\|_{T^1(K)}+\varepsilon \mu(F),
\end{equation*}
which, after letting $\varepsilon\to 0$ and then $a\to0$, proves \eqref{thefirstdirection}.\\

Now, to finish the first direction, note that for an arbitrary $\mu\in T^1(K)$ and $f\in S^1(K)$, we have that $|\mu|\in T^1(K)$ and $|f|\in S^1(K)$, with norms remaining unchanged. Therefore, the integral  
\[
\int_{\upperhalf}fd\mu
\]
is absolutely convergent, and the triangle inequality implies  
\[
\left|\int_{\upperhalf}fd\mu\right|\le\int_{\upperhalf}|f|d|\mu|\lesssim \|f\|_{S^1(K)}\|\mu\|_{T^1(K)},
\]
which completes the proof of this direction.\\  

For the second direction, we have to show that any bounded linear functional $l$ on $S_0^1(K)$ comes from a unique measure $\mu$ in $T^1(K)$, i.e.,  
\begin{equation}
	l(f)=\int_{\upperhalf}fd\mu, \quad f\in S_0^1(K).
\end{equation}

To this aim, note that for every compact subset $F\subset \upperhalf$, there exists a constant $c(F)$ such that for any continuous function $f$ supported in $F$, we have  
\[
|l(f)|\le \|l\|_{S_0^1(K)^*}\|f\|_{S_0^1(K)}\le  c(F)\|l\|_{S_0^1(K)^*} \sup_{x\in F}|f(x)|.
\]
To see this, let $s$ and $r$ be such that  
\[
\{t:(x,t)\in F\}\subset [0,s], \quad \{x: (x,t)\in F\}\subset B(0,r).
\]
Then, any disjoint collection of $K$-balls $B(x,t)$ with $(x,t)\in F$ lies entirely in $B(0,r+s)$, and therefore, the above inequality holds with $c(F)=|B(0,r+s)|$. Now, the Riesz-Markov-Kakutani representation theorem implies that there exists a Radon measure $\mu$ such that  
\begin{equation}\label{functionaldef}
	l(f)=\int_{\upperhalf}fd\mu, \quad f\in C_c(\upperhalf).
\end{equation}
Now, the above identity implies that $\mu$ is unique, and it remains to show that $\mu\in T^1(K)$. To do this, consider functions of the form  
\begin{equation}\label{loneatom}
	f(x,t):=a(x,t)\int_{B(x,t)}g(y)dy, \quad a\in C_c(\upperhalf), \quad |a|\le 1, \quad \lonernn{g}=1,
\end{equation}
which are continuous, compactly supported, belong to $S_0^1(K)$, and have norm at most $1$. Therefore, from Fubini, we must have  
\[
l(f)=\int_{\upperhalf}f(x,t)d\mu(x,t)=\int_{\rn}g(y)\int_{\Gamma_y}a(x,t)d\mu(x,t),
\]
which implies that  
\[
\left|\int_{\rn}g(y)\int_{\Gamma_y}a(x,t)d\mu(x,t)\right|\le \|l\|_{S_0^1(K)^*},
\]
and since this holds for an arbitrary $g$ in the unit ball of $\lonern$, we obtain  
\[
\left|\int_{\Gamma_y}a(x,t)d\mu(x,t)\right|\le \|l\|_{S_0^1(K)^*} \quad \text{a.e.} \quad y\in\rn.
\]

Here, the negligible set depends on $\mu$ and $a$. To eliminate this dependence on $a$, note that the space of continuous functions with compact support is separable. Thus, after discarding a set of measure zero, the above inequality still holds for a countable dense subset of the unit ball of $C_c(\upperhalf)$. By continuity of the integral, it then extends to all functions in the unit ball of $C_c(\upperhalf)$. Now, for each $y$ where this holds, we obtain  
\[
|\mu|(\Gamma_y)\le \|l\|_{S_0^1(K)^*},
\]
which implies that $\|\mu\|_{T^1(K)}\le \|l\|_{S_0^1(K)^*}$. This, combined with \eqref{functionaldef} and the first direction, yields $\|l\|_{S_0^1(K)^*}\simeq \|\mu\|_{T^1(K)}$, completing the proof.

\end{proof}
Now, we proceed to prove Theorem \ref{preduals}. To this end, we recall the following general atomic decomposition theorem, which can be found in \cite{MR1225511}.

\begin{theorem}\label{atomiclemma}
	Let $V \subset X$ be a bounded subset of a Banach space $X$, symmetric around the origin, and such that for every $l \in X^*$,
	\[
	\|l\|_{X^*} \simeq \sup_{x \in V} |l(x)|.
	\]
	Then, for every $x \in X$, there exists a sequence of positive real numbers $\{\lambda_i\}_{i \geq 1}$ such that
	\[
	x = \sum_{i \geq 1} \lambda_i v_i, \quad \|x\|_X \simeq \sum_{i \geq 1} \lambda_i, \quad v_i \in V, \quad i \geq 1.
	\]
	In the above, the series converges in the norm topology of $X$.
\end{theorem}

\begin{proof}[Proof of Theorem \ref{preduals}]
Let \( V^{p'}(K) \) be the space of all functions \( f \in L^{p'}_{\text{loc}}(\upperhalf) \), for which
\[
\|f\|_{V^{p'}(K)} := \sup_{\|g\|_{T^p(K)} \le 1} \int_{\upperhalf} |f(x,t) g(x,t)| \, dx \, dt < \infty.
\]
The space \( V^{p'}(K) \) is a linear space, and the above quantity defines a norm on it. Our first aim is to show that \( V^{p'}(K) \) is a Banach space, and to this aim, let \( O \) be a compact region with non-empty interior in \( \upperhalf \) and note that \( O \) is covered by finitely many cones \( \Gamma_x \) with \( x \in \rn \). Therefore, for any function \( g \) supported in \( O \), we have
\[
\|g\|_{T^p(K)} \le \|g\|_{L^p(O)} \le c(O) \|g\|_{T^p(K)}.
\]
From this, we get
\begin{equation}\label{localinclusion}
	\|f\|_{L^{p'}(O)} \le c(O) \|f\|_{V^{p'}(K)},
\end{equation}
for all \( f \) with support in \( O \). This shows that \( V^{p'}(K) \) is embedded in \( L^{p'}_{\text{loc}}(\upperhalf) \), and thus every Cauchy sequence \( \{f_i\} \) in \( V^{p'}(K) \) converges to a function \( f \in L^{p'}_{\text{loc}}(\upperhalf) \). Now, let \( \|g\|_{T^p(K)} \le 1 \), and note that Fatou's lemma implies
\[
\int_{\upperhalf} |f g| \le \liminf_{i \to \infty} \int_{\upperhalf} |f_i g| \, dx \, dt \le \sup_i \|f_i\|_{V^{p'}(K)},
\]
which implies that \( f \in V^{p'}(K) \). Another application of Fatou's lemma gives
\[
\int_{\upperhalf} |f_l - f| |g| \le \liminf_{i \to \infty} \int_{\upperhalf} |f_l - f_i| |g| \, dx \, dt \le \liminf_{i \to \infty} \|f_l - f_i\|_{V^{p'}(K)},
\]
which, after taking the supremum over \( g \) and letting \( l \to \infty \), shows that
\[
\lim_{l \to \infty} \|f_l - f\|_{V^{p'}(K)} = 0,
\]
proving that \( V^{p'}(K) \) is complete.\\

Now, let \( U^{p'}(K) \) be the closure of \( C_c(\upperhalf) \) in \( V^{p'}(K) \), which is a Banach space, and we will show that \( T^p(K) \) is its dual. To do this, we note that any \( g \in T^p(K) \) gives the bounded linear functional
\[
l_g(f) := \int_{\upperhalf} f(x,t) g(x,t) \, dx \, dt,
\]
on \( U^{p'}(K) \) with \( \|l_g\|_{U^{p'}(K)^*} \le \|g\|_{T^p(K)} \), and it remains to show that for any continuous linear functional \( l \) on \( U^{p'}(K) \), there exists a unique function \( g \) such that \( l = l_g \), with \( \|g\|_{T^p(K)} \le \|l\|_{U^{p'}(K)^*} \). In doing so, we follow the same line of argument as in the previous proof, i.e., from \eqref{localinclusion} we get that \( l \) defines a bounded linear functional on \( L^{p'}(O) \) for any compact set with non-empty interior in \( \upperhalf \). This implies that there exists a function \( g \in L^p_{\text{loc}}(\upperhalf) \) such that
\[
l(f) = \int_{\upperhalf} f(x,t) g(x,t) \, dx \, dt, \quad f \in C_c(\upperhalf),
\]
which in particular shows that \( g \) is unique. Then, let
\begin{equation}\label{lpatoms}
	f(x,t) = a(x,t) \int_{B(x,t)} h(y) \, dy, \quad a \in C_c(\upperhalf), \quad \|a\|_{L^{p'}(\upperhalf)} \le 1, \quad \|h\|_{L^1(\rn)} = 1,
\end{equation}
and note that \( \|f\|_{U^{p'}(K)} \le 1 \). To see this, take \( b \) in the unit ball of \( T^p(K) \), and note that from Fubini we have
\[
\int_{\upperhalf} \left| \int_{B(x,t)} h(y) \, dy \, a(x,t) \right| |b(x,t)| \, dx \, dt \le \int_{\rn} |h(y)| \int_{\Gamma_y} |a(x,t)| |b(x,t)| \, dx \, dt,
\]
which, after using Hölder's inequality, gives us
\[
\int_{\upperhalf} |f b| \le 1,
\]
as desired. Now, a similar computation shows
\[
l(f) = \int_{\upperhalf} h(y) \int_{\Gamma_y} a(x,t) g(x,t) \, dx \, dt \le \|l\|_{U^{p'}(K)^*},
\]
and therefore,
\[
\left| \int_{\Gamma_y} a(x,t) g(x,t) \, dx \, dt \right| \le \|l\|_{U^{p'}(K)^*}, \quad \text{a.e.} \quad y\in\rn.
\]
This time we use the fact that \( L^{p'}(\upperhalf) \) is separable, and thus, after neglecting a null set and using continuity, we get
\[
\|g\|_{L^p(\Gamma_y)} \le \|l\|_{U^{p'}(K)^*}, \quad \text{a.e.} \quad y \in \rn,
\]
which means that \( \|g\|_{T^p(K)} \le \|l\|_{U^{p'}(K)^*} \), and this completes the proof of this direction.\\

Finally, to prove the claimed decomposition, we appeal to Theorem \ref{atomiclemma}, and note that the two sets of functions defined in \eqref{lpatoms} and \eqref{loneatom} satisfy the conditions of \( V \), and therefore the decompositions follow. The proof is now complete.
\end{proof}

\section{Dyadic Analogues}
As mentioned in the introduction, our aim in this section is to characterize the dual \( S^p(K) \) spaces in the dyadic setting. Here, for simplicity, we restrict ourselves to dyadic intervals on \( [0,1) \), and for \( 1 \le p < \infty \), we let \( X^{p,\infty} \) be the space of all sequences of numbers \( f = \{ f_I \}_{I \in \mathcal{D}} \), indexed by half-open dyadic intervals \( \mathcal{D} \) in \( [0,1) \), such that
\[
\|f\|_{X^{p,\infty}} := \big( \sup_{\mathcal{C}} \sum_{I \in \mathcal{C}} |f_I|^p \big)^{\frac{1}{p}} < \infty,
\]
where, as usual, the above supremum is taken over all finite disjoint collections of dyadic intervals in \( [0,1) \). Quite similarly to the continuous case, \( X^{p,\infty} \) is a Banach space containing all sequences of finite support. Here, we take advantage of the structure of dyadic intervals and give another interpretation of the above norm. For each \( j \ge 0 \), let \( \mathcal{G}_j \) be the \( j \)-th generation of dyadic intervals, i.e., dyadic intervals with length \( 2^{-j} \). Then, suppose \( f = f^{i} \) is supported on \( \mathcal{D}_i = \cup_{1 \le j \le i} \mathcal{G}_j \), and note that to compute the above norm, we may follow a greedy algorithm from bottom to top. More precisely, for each \( I \in \mathcal{G}_{i-1} \) with its two children \( I_l \) and \( I_r \) in \( \mathcal{G}_i \), replace \( f_I \) with
\[
\max\left\{ \left( |f_{I_r}|^p + |f_{I_l}|^p \right)^{\frac{1}{p}}, |f_I| \right\},
\]
and set \( f_{I_l} \) and \( f_{I_r} \) to 0. Then, we get a new sequence \( f^{i-1} \) which is supported on \( \mathcal{D}_{i-1} \), and with the property that \( \| f^{i-1} \|_{X^{p,\infty}} = \| f^{i} \|_{X^{p,\infty}} \). Next, repeat the process with \( f^{i-1} \) and get \( f^{i-2} \), and so on. The process terminates at the \( i \)-th step, where we have a constant \( f^0 = f^0_{[0,1)} \), which is exactly \( \| f \|_{X^{p,\infty}} \).\\

The above procedure has a natural generalization to a two-parameter scale of spaces, \(X^{p,q}\), which contains the dual of \(X_0^{p,\infty}\), the closure of finitely supported sequences in \(X^{p,\infty}\). To this aim, let \(1 \le p,q \le \infty\), and \(g = g^{i}\) be supported on \(\mathcal{D}_i\), then for each \(I \in \mathcal{G}_{i-1}\) replace \(g_I\) with
\[
((g_{I_r}, g_{I_l})_{p}, g_I)_{q} = \left( (|g_{I_r}|^{p} + |g_{I_l}|^{p})^{\frac{q}{p}} + |g_I|^q \right)^{\frac{1}{q}},
\]
and \(g_{I_r}, g_{I_l}\) with 0 to get \(g^{i-1}\). Here, we use the standard convention that for \(p = \infty\) or \(q = \infty\) we take the maximum. Then, continuing this process results in a constant, \(g^{0} = g^0_{[0,1)}\), at the \(i\)-th step which we denote as \(\|g\|_{X^{p,q}}\). To describe this norm for an arbitrary sequence \(g\), we set 
\begin{equation}\label{gidef}
	g_{i,I} := g_I, \quad I \in \mathcal{D}_i, \quad g_{i,I} := 0, \quad I \notin \mathcal{D}_i, \quad i \ge 0,
\end{equation}
and then
\begin{equation}\label{normofg}
	\|g\|_{X^{p,q}} := \sup_{i \ge 0} \|g_i\|_{X^{p,q}}.
\end{equation}
Now, we define the space \(X^{p,q}\) as the space of all sequences $g = \{g_I\}_{I \in \mathcal{D}}$, for which the above quantity is finite. This scale of spaces contains the dyadic analogue of \(T^q(K)\) as well. Since when \(p = \infty\) and \(1 \le q < \infty\), we have
\begin{equation}\label{pointincomom}
	\|g\|_{X^{\infty, q}} = \sup_{x \in [0,1]} \left( \sum_{I\in\mathcal{D}} |g_I|^q 1_I(x) \right)^{\frac{1}{q}},
\end{equation}
which is the dyadic analogue of \(T^q(K)\)-norm.\\

Before we proceed any further, let us show that this quantity is a norm, which makes \(X^{p,q}\) into a Banach space. First, observe that the above quantity is homogeneous and positive definite, and it remains to show that the triangle inequality holds as well. The first thing to note is that it is enough to prove the claim for sequences with finite support and non-negative coefficients. Then, as the induction assumption, suppose the triangle inequality holds for sequences supported on \(\mathcal{D}_{i-1}\), and now take \(f\) and \(g\) with support on \(\mathcal{D}_i\). Then, note that for any two sequences \(a\) and \(b\) supported on \(\mathcal{D}_i\), we have
\[
\|a\|_{X^{p,q}} = \|a^{i-1}\|_{X^{p,q}}, \quad \|a\|_{X^{p,q}} \le \|b\|_{X^{p,q}} \quad \text{if} \quad |a_I| \le |b_I|, \quad I \in \mathcal{D}.
\]

Now, since
\[
(f+g)_I^{i-1} \le f_I^{i-1} + g_I^{i-1}, \quad I \in \mathcal{G}_{i-1},
\]
the above two observations imply
\[
\|f+g\|_{X^{p,q}} = \|(f+g)^{i-1}\|_{X^{p,q}} \le \|f^{i-1}\|_{X^{p,q}} + \|g^{i-1}\|_{X^{p,q}} = \|f\|_{X^{p,q}} + \|g\|_{X^{p,q}},
\]
which is the desired inequality. The completeness follows easily from
\[
|f_I| \le \|f\|_{X^{p,q}}, \quad I \in \mathcal{D},
\]
and a standard argument, which we leave to the interested reader.\\

In the next theorem, we give a characterization of the dual of \(X_0^{p,q}\).
\begin{theorem}\label{xdual}
	Let \(X_0^{p,q}\) be the closure of the space of sequences with finite support in \(X^{p,q}\). Then we have
	\begin{align}\label{unify}
		&(X_0^{p,q})^* = X^{p',q'}, \quad 1 \le p, q \le \infty,
	\end{align}
	with the natural pairing 
	\[
	\left\langle f, g \right\rangle := \sum_{I \in \mathcal{D}} f_I g_I.
	\]

\end{theorem}
\begin{proof}
The first step is to show that
\begin{equation}\label{dyadicdual}
	|\sum_{I \in \mathcal{D}} f_I g_I| \le \|f\|_{X^{p,q}} \|g\|_{X^{p',q'}}, \quad f \in X^{p,q}, \quad g \in X^{p',q'}, \quad 1 \le p, q \le \infty.
\end{equation}
Now, it is enough to prove the above inequality for \(f\) and \(g\) with finite support and non-negative coefficients. Therefore, we may appeal to induction again, and assume that the above holds for any two sequences supported on \(\mathcal{D}_{i-1}\), and \(f\) and \(g\) are now supported on \(\mathcal{D}_i\). We have
\[
\sum_{I \in \mathcal{D}_i} f_I g_I = \sum_{I \in \mathcal{D}_{i-2}} f_I g_I + \sum_{I \in \mathcal{G}_{i-1}} f_I g_I + f_{I_r} g_{I_r} + f_{I_l} g_{I_l},
\]
and then two applications of Hölder's inequality give us
\[
\sum_{I \in \mathcal{G}_{i-1}} f_I g_I + f_{I_r} g_{I_r} + f_{I_l} g_{I_l} \le \sum_{I \in \mathcal{G}_{i-1}} f_I^{(i-1)} g_I^{(i-1)},
\]
which, after plugging in the previous identity, implies
\[
\sum_{I \in \mathcal{D}_i} f_I g_I \le \sum_{I \in \mathcal{D}_{i-1}} f_I^{(i-1)} g_I^{(i-1)}.
\]
Therefore, from the induction assumption and the fact that
\[
\|f\|_{X^{p,q}} = \|f^{(i-1)}\|_{X^{p,q}}, \quad \|g\|_{X^{p',q'}} = \|g^{(i-1)}\|_{X^{p',q'}},
\]
we obtain
\[
\sum_{I \in \mathcal{D}_i} f_I g_I \le \|f\|_{X^{p,q}} \|g\|_{X^{p',q'}},
\]
which is the desired inequality. The inequality \eqref{dyadicdual} shows that any sequence \(g \in X^{p',q'}\) defines a bounded linear functional \(l_g\), with the natural pairing
\[
l_g(f) := \sum_{I \in \mathcal{D}} f_I g_I,
\]
on \(X^{p,q}\), whose operator norm is no more than \(\|g\|_{X^{p',q'}}\). As the next step, we show that for every sequence \(g\) we have
\begin{equation}\label{upperbound}
	\|g\|_{X^{p',q'}} \le \sup_{\|f\|_{X^{p,q}} \le 1} \big|\sum_{I \in \mathcal{D}} f_I g_I\big|, \quad 1 \le p, q \le \infty.
\end{equation}
To do this, recall \eqref{normofg} and note that it is enough to prove the above inequality for \(g\) with finite support. Therefore, we may use induction again, and as the induction assumption, we assume that for every \(g\) supported on \(\mathcal{D}_{i-1}\), \eqref{upperbound} holds. Now, we assume \(g\) is supported on \(\mathcal{D}_i\). For \(I \in \mathcal{D}_{i-1}\), we may choose \(f'_I\), \(f'_{I_r}\), and \(f'_{I_l}\) such that
\[
f'_I g_I + f'_{I_r} g_{I_r} + f'_{I_l} g_{I_l} = \left( (g_{I_r}, g_{I_l})_{p'}, g_I \right)_{q'}, \quad \left( (f'_{I_r}, f'_{I_l})_{p}, f'_I \right)_{q} = 1,
\]
simply because the two norms appearing in the above expressions are dual to each other.
 Then, for any \(f\) supported on \(\mathcal{D}_{i-1}\), we have
 \begin{align*}
 	&\sum_{I \in \mathcal{D}_{i-1}} g_I^{i-1} f_I = \sum_{I \in \mathcal{D}_{i-2}} g_I f_I + \sum_{I \in \mathcal{G}_{i-1}} \left( (g_{I_r}, g_{I_l})_{p'}, g_I \right)_{q'} f_I = \\
 	&\sum_{I \in \mathcal{D}_{i-2}} g_I f_I + \sum_{I \in \mathcal{G}_{i-1}} f_I f'_I g_I + f_I f'_{I_r} g_{I_r} + f_I f'_{I_l} g_{I_l},
 \end{align*}
 which implies that
 \[
 \|g\|_{X^{p,q}} = \|g^{i-1}\|_{X^{p',q'}} \le \sup_{\|f\|_{X^{p,q}} \le 1} \big| \sum_{I \in \mathcal{D}_{i-1}} g_I^{i-1} f_I \big| \le \sup_{\|f\|_{X^{p,q}} \le 1} \big| \sum_{I \in \mathcal{D}} f_I g_I \big|,
 \]
 the desired inequality.\\

Now, to complete the proof, note that if \( l \) is a bounded linear functional on \( X_0^{p,q} \), we must have 
\[
l(f) = \sum_{I \in \mathcal{D}} f_I g_I,
\]
for all \( f \in X^{p,q} \) with finite support and a unique sequence of numbers \( g \). Therefore, \eqref{upperbound} shows that
\[
\|g\|_{X^{p',q'}} \le \|l\|, \quad 1 \le p, q \le \infty,
\]
which, together with \eqref{dyadicdual}, implies that
\[
(X_0^{p,q})^* = X^{p',q'}, \quad 1 \le p, q \le \infty.
\]
The proof is now complete.
\end{proof}
The duality relation \eqref{unify} unifies Theorems \ref{duality1} and \ref{preduals} in the dyadic setting, and quite similarly to Theorem \ref{preduals}, we may use Theorem \ref{atomiclemma} and establish an atomic decomposition for the preduals of \( X_0^{\infty,q} \).

\begin{corollary}
	Let \( 1 \le q < \infty \), then every sequence \( f \) in \( X_0^{1,q'} \) can be written as
	\[
	f = \sum_{i \ge 0} \lambda_i f_i, \quad \|f\|_{X^{1,q'}} \simeq \sum_{i \ge 0} \lambda_i, \quad \|f_i\|_{X^{1,q'}} = 1, \quad i \ge 0,
	\]
	such that \( \lambda_i > 0 \) and \( \{f_i\} \) is a sequence whose coefficients are supported on a nested collection of dyadic intervals in \( [0,1) \). Also, for \( q = 1 \), the coefficients of \( f_i \) are only \( +1 \) and \( -1 \).
\end{corollary}

\begin{proof}
	In Theorem \ref{atomiclemma}, let $X= X_0^{1,q'}$, and be  \( V \) be the subset of \( X_0^{1,q'} \) consisting of sequences \( g \) supported on dyadic intervals with a point like \( x \in [0,1] \) in common, and satisfying
	\[
	\sum_{x \in I} |g_I|^{q'} = 1, \quad 1 < q < \infty, \quad g_I = \pm 1, \quad x \in I, \quad q =1.
	\]
	Then, from Theorem \ref{xdual} we have $X^*=X^{\infty, q}$, which after recalling \eqref{pointincomom} implies that \( V \) satisfies the condition of Theorem \ref{atomiclemma}, and the claim follows.
\end{proof}

The above decomposition for \( q = 1 \) can be rephrased as the following: To construct a sequence in \( X_0^{1,\infty} \), choose a sequence of points \( \{x_i\}_{i\ge1} \) in \( [0,1] \) and a sequence of positive numbers \( \{\lambda_i\}_{i\ge1} \) in \( l_1 \), then let
\[
f_I = \sum_{x_i \in I} \epsilon_{i,I} \lambda_i, \quad \epsilon_{i,I} = \pm 1.
\]
This way, we may construct all sequences in \( X_0^{1,\infty} \), and in addition, we may choose \( \{x_i\}_{i\ge1} \) and \( \{\lambda_i\}_{i\ge1} \) such that
\[
\|f\|_{X^{1,\infty}} \simeq \sum_{i \ge 0} \lambda_i.
\]
Also, since for \( 1 \le p < \infty \), and \( f \in X^{p,\infty} \), we have \( f^{p} \in X^{1,\infty} \), we conclude that for such \( f \), we have
\[
f_I = \pm \big| \sum_{x_i \in I} \pm \lambda_i \big|^{\frac{1}{p}}, \quad I \in \mathcal{D}, \quad \|f\|_{X^{p,\infty}} \simeq \big( \sum_{i \ge 0} \lambda_i \big)^{\frac{1}{p}}.
\]
Now, let us briefly discuss the dyadic \( \text{JN}_p \) space on the unit interval and its relation to the space \( X^{p,\infty} \). Recall that for \( 1 < p < \infty \), a locally integrable function \( f \) belongs to \( \text{JN}^d_p[0,1] \) if
\[
\|f\|_{\text{JN}^d_p[0,1]} := \|\{\osc{f}{I}|I|^{\frac{1}{p}}\}_{I \in \mathcal{D}}\|_{X^{p,\infty}} < \infty.
\]
See \cite{MR4546791} for more information on this space. For \( 2 < p < \infty \), replacing \( \osc{f}{I} \) with the \( L^2 \)-mean oscillation
\[
\text{osc}_2(f,I)^2 := \ave{|f - f_I|^2}{I},
\]
results in an equivalent norm, i.e.
\[
\|f\|_{\text{JN}^d_p[0,1]} \simeq \|\{\text{osc}_2(f,I)|I|^{\frac{1}{p}}\}_{I \in \mathcal{D}}\|_{X^{p,\infty}}.
\]
See \cite{MR3806819}. The advantage of using \( L^2 \)-mean oscillations lies in the fact that they have a simple expression in terms of Haar coefficients. Recall that the $L^\infty$-normalized Haar basis of \( L^2[0,1] \) consists of functions
\[
h_I = 1_{I_r} - 1_{I_l}, \quad I \in \mathcal{D} \setminus {[0,1)}, \quad h_{[0,1)} := 1_{[0,1)},
\]
which form an orthogonal basis for \( L^2[0,1] \). Thus, any function \( f \in L^2[0,1] \) can be uniquely written as
\[
f = \sum_{I \in \mathcal{D}} f_I h_I, \quad f_I := |I|^{-1}\left\langle f, h_I \right\rangle, \quad I \in \mathcal{D}.
\]
Therefore, for any \( I \in \mathcal{D} \), we have
\[
\text{osc}_2(f,I)^2 = \frac{1}{|I|}\sum_{J \subseteq I} f_J^2|J|,
\]
which implies that
\begin{equation}\label{haardef}
	f_I^2|I| = |I|\text{osc}_2(f,I)^2 - |I_r|\text{osc}_2(f,I_r)^2 - |I_l|\text{osc}_2(f,I_l)^2\ge 0,
\end{equation}
from which we conclude that
\begin{equation}\label{theaboveintw}
	g_I \ge  2^{\frac{1}{p}-\frac{1}{2}}(g_{I_r}^2 + g_{I_l}^2)^{\frac{1}{2}}, \quad g_I = \text{osc}_2(f,I)|I|^{\frac{1}{p}}, \quad I \in \mathcal{D}.
\end{equation}
The above inequality, together with \eqref{haardef}, shows that for $2<p<\infty$, there exists a correspondence between functions \( f \in \text{JN}^d_p[0,1] \) and sequences \( g \in X^{p,\infty} \) satisfying the property on the left-hand side of \eqref{theaboveintw}. For if we have such \( g \), we can use \eqref{haardef} to easily determine the Haar coefficients of \( f \) up to a choice of signs. However, it is unfortunate that we were unable to characterize this subspace of \( X^{p,\infty} \).\\

Next, as an application of Theorem \ref{xdual}, we characterize some bounded Haar multipliers from \( SL^{\infty}[0,1] \), the space of functions with a bounded square function, into \( L^2[0,1] \). Recall that
\[
\|f\|_{SL^{\infty}[0,1]} := \|S(f)\|_{L^{\infty}[0,1]}, \quad S(f)(x) := \big( \sum_{I\in\mathcal{D}} f_I^2 1_I(x) \big)^{\frac{1}{2}}.
\]
See \cite{MR2157745} for more on this space. Now, let \( T_a \) be a Haar multiplier on \( L^2[0,1] \), i.e.,
\begin{equation}\label{haarmultiplier}
	T_a(f)(x) = \sum_{I \in \mathcal{D}} \frac{a_I}{\sqrt{|I|}} f_I h_I(x), \quad a = \{a_I\}_{I \in \mathcal{D}},
\end{equation}
where we assume that only finitely many coefficients of \( a \) are non-zero. Then, we have
\[
\|T_a\|^2_{SL^{\infty} \to L^2[0,1]}=\sup_{\|f\|_{SL^{\infty}[0,1]}=1}\|T(f)\|_{L^2[0,1]}^2 =\sup_{\|f\|_{SL^{\infty}[0,1]}=1} \sum_{I \in \mathcal{D}} a_I^2 f_I^2,
\]
which together with Theorem \ref{xdual} and the fact that
\begin{equation}\label{slintysupmax}
	\|f\|_{SL^{\infty}[0,1]}= \|\{f_I\}_{I \in \mathcal{D}}\|_{X^{\infty,2}}=\|\{f^2_I\}_{I \in \mathcal{D}}\|^{\frac{1}{2}}_{X^{\infty,1}},
\end{equation}
 implies 
\[
\|T_a\|_{SL^{\infty} \to L^2[0,1]}= \|\{a_I^2\}_{I \in \mathcal{D}}\|_{X^{1,\infty}}^{\frac{1}{2}} = \|a\|_{X^{2,\infty}}.
\]

\begin{proposition}
Let \( a = \{a_I\}_{I \in \mathcal{D}} \) be a sequence with finite support. Then, for the Haar multiplier \( T_a \) defined as in \eqref{haarmultiplier}, we have
\[
\|T_a\|_{SL^{\infty} \to L^2[0,1]}= \|a\|_{X^{2,\infty}}.
\]
\end{proposition}

Finally, a well-known property of functions in \( SL^{\infty}[0,1] \) is that for an absolute constant \( c \), the following inequality holds:
\[
\int_{[0,1]} \exp \left( \frac{c |f(x)|^2}{\|f\|_{SL^{\infty}}^2} \right) \, dx \lesssim 1, \quad \int_{[0,1]} f = 0,
\]
which is equivalent to
\[
\|f\|_{e^{L^2}[0,1]} \lesssim \|f\|_{SL^{\infty}}, \quad f_{[0,1)} = 0,
\]
where \( e^{L^2} \) is used to denote the Orlicz norm \cite{MR2157745}. 
The above inequality is nothing but the boundedness of the map
\[
\mathcal{S} : \left\{ g \in X_0^{\infty,2} : g_{[0,1)} = 0 \right\} \to e^{L^2}[0,1], \quad \mathcal{S} \big( \{g_I\}_{I \in \mathcal{D}} \big) = \sum_{I \in \mathcal{D}} g_I h_I.
\]
Therefore, after taking the adjoint, applying Theorem~\ref{xdual}, and using the duality fact that
\[
L(1 + \log^+)^{\frac{1}{2}} L[0,1]^* = e^{L^2}[0,1],
\]
we obtain the following result:

\begin{proposition}
For any \( f \) with \( \int_{[0,1)} f = 0 \), and finite Haar expansion we have
\[
\|\{f_I\}_{I \in \mathcal{D}}\|_{X^{1,2}} \lesssim \|f\|_{L(1 + \log^+)^{\frac{1}{2}} L[0,1]}.
\]
\end{proposition}

\section{acknowledgments}
I am indebted to Professor Galia Dafni for posing an inspiring question that led to the material in this paper and for hours of valuable discussions. I would also like to thank Professor Riikka Korte for an interesting discussion on the subject.


\end{document}